\documentclass[a4paper,12pt]{amsart}

\makeindex

\usepackage{amsfonts,graphics,amsmath,amsthm,amsfonts,amscd,amssymb,amsmath,latexsym,euscript, enumerate}
\usepackage{epsfig,url}
\usepackage{flafter}
\usepackage[all,cmtip,line]{xy}
\usepackage{array}
\usepackage[english]{babel}
\usepackage{overpic}
\usepackage{subfig}
\usepackage{multirow}

\usepackage{wrapfig}
\usepackage[shortlabels]{enumitem}
\setlist[enumerate, 1]{1\textsuperscript{o}}

\usepackage[usenames,dvipsnames,svgnames,table]{xcolor}
\usepackage{graphicx}
\usepackage[bookmarks=true,bookmarksnumbered=true,
 pdftitle={},
 pdfsubject={},
 pdfauthor={},
 pdfkeywords={TeX, dvipdfmx, hyperref, color,},
 colorlinks=true, linkcolor=RoyalBlue, citecolor=Emerald]
 {hyperref}

\newtheorem{theorem}{Theorem}[section]
\newtheorem{lemma}[theorem]{Lemma}

\newtheorem{question}[theorem]{Question}

\theoremstyle{corollary}
\newtheorem{corollary}[theorem]{Corollary}

\theoremstyle{definition} % italic or bold etc.

\newtheorem{definition-lemma}[theorem]{Definition-Lemma}

\newtheorem{example}[theorem]{Example}

\theoremstyle{remark}
\newtheorem{remark}[theorem]{Remark}

\numberwithin{equation}{section}

\newcommand{\C}{\mathbb{C}}
\newcommand{\R}{\mathbb{R}}
\newcommand{\Z}{\mathbb{Z}}

\newcommand{\Q}{\mathbb{Q}}

\DeclareMathOperator{\im}{Image}

\DeclareMathOperator{\vol}{vol}
\DeclareMathOperator{\val}{val}

\newcommand{\okbd}{\Delta}
\newcommand{\okval}{\Delta^{\val}}
\newcommand{\oklim}{\Delta^{\lim}}

\def\BDPP{\operatorname{BDPP}}

\def\Spec{\operatorname{Spec}}

\def\mult{\operatorname{mult}}

\def\codim{\operatorname{codim}}
\def\var{\operatorname{Var}}

\title[Subadditivity of Okounkov bodies]
{On subadditivity of Okounkov bodies for \\algebraic fiber spaces}

\begin{document}

\author{Sung Rak Choi}
\address{Department of Mathematics, Yonsei University, 50 Yonsei-ro, Seodaemun-gu, Seoul 03722, Republic of Korea}
\email{sungrakc@yonsei.ac.kr}

\author{Jinhyung Park}
\address{Department of Mathematics, Sogang University, 35 Baekbeom-ro, Mapo-gu, Seoul 04107, Republic of Korea}
\email{parkjh13@sogang.ac.kr}

%\thanks{}

%\subjclass[2010]{14C20, 14D06.}
\date{\today}

\begin{abstract}
The purpose of this paper is to establish a subadditivity theorem of Okounkov bodies for algebraic fiber spaces. As applications, we obtain a product formula of the restricted canonical volumes for algebraic fiber spaces and a sufficient condition for an algebraic fiber space to be birationally isotrivial in terms of Okounkov bodies when a general fiber is of general type. Furthermore, we also prove the subadditivity of the numerical Iitaka dimensions for algebraic fiber spaces, and this confirms some numerical variants of the Iitaka conjecture. We hope that our results would provide a new approach toward the Iitaka conjecture.
\end{abstract}

\maketitle
%\tableofcontents

%%%%%%%%%%%%%%%%%%%%%%%%%%%%%%%%%%%%%%%%%%%%%%%%%%%%%
\section{Introduction}
%%%%%%%%%%%%%%%%%%%%%%%%%%%%%%%%%%%%%%%%%%%%%%%%%%%%%

Throughout the paper, we work over the field $\C$ of complex numbers. In this paper, every divisor is assumed to be a $\Q$-Cartier $\Q$-divisor unless otherwise stated, and an \emph{algebraic fiber space} is a surjective morphism $f \colon X\to Y$ between smooth projective varieties $X,Y$ with connected fibers. We denote by $F$ a general fiber of $f$. The intertwined connection among $X,Y,F$ is still mysterious and has yet to be studied further.

\medskip

An \emph{Okounkov body} of a divisor is a convex subset in Euclidean space introduced by Kaveh--Khovanskii \cite{KK} and Lazarsfeld--Musta\c{t}\u{a} \cite{lm-nobody} independently. It is well known that the Okounkov bodies encode various asymptotic invariants of divisors (see Section \ref{sec:prelim} for the relevant definitions). Our experience indicates that the Okounkov bodies are extremely useful in studying the subadditivity properties of the Iitaka dimensions and the volumes of divisors. The aim of this paper is to unveil some relations on the positivity properties of the canonical divisors $K_X,K_Y, K_F$ in a fixed framework provided by the Okounkov bodies.

\medskip

First of all, we recall one of the most important problems in birational geometry about algebraic fiber spaces, which still remains open since 1972 (see \cite{Fujino-book} for the comprehensive account for the Iitaka conjecture).

\medskip

\noindent\textbf{Iitaka Conjecture.}
\emph{Let $f\colon X\to Y$ be an algebraic fiber space with general fiber $F$. Then the following inequality for the Kodaira dimensions holds:}
$$
\kappa(K_X) \geq \kappa(K_Y)+\kappa(K_F).
$$

\medskip

Along with the Kodaira dimension, the volume function also provides a refined measurement of the positivity of canonical divisors. Regarding the relation on the volumes of canonical divisors $K_X, K_Y, K_F$, the following product formula of volumes was obtained by Kawamata \cite{K2} under the assumption that $X, Y, F$ are of general  type:
\begin{equation}\label{*}
\frac{\vol_X(K_X)}{\dim X!} \geq \frac{\vol_Y(K_Y)}{\dim Y!} \cdot \frac{\vol_F(K_F)}{\dim F!}.
\end{equation}
In \cite{CJPW}, the product formula (\ref{*}) was interpreted and reproved from an alternative perspective in terms of Okounkov bodies of canonical divisors. More precisely, the formula (\ref{*}) can be obtained as an immediate consequence of the following ``subadditivity'' of the Okounkov bodies of canonical divisors:
\begin{equation}\label{**}
\okbd_{X_\bullet}(K_X) \supseteq \okbd_{Y_\bullet}(K_Y) + \okbd_{F_\bullet}(K_F)
\end{equation}
where $X_\bullet$ is a \emph{fiber-type} admissible flag on $X$ associated to admissible flags $Y_\bullet, F_\bullet$ on $Y, F$, respectively with $F = f^{-1}(Y_{\dim Y})$, i.e.,
$$
X_i = \begin{cases} f^{-1}(Y_i) & \text{if $0 \leq i \leq \dim Y$} \\
F_{i-\dim Y} & \text{if $\dim Y < i \leq \dim X$}. \end{cases}
$$
By identifying $\Delta_{Y_\bullet}(K_Y)$ with $\Delta_{Y_\bullet}(K_Y)\times\{0\}^{\dim F}$ and  $\Delta_{F_\bullet}(K_F)$ with $\{0\}^{\dim Y}\times\Delta_{F_\bullet}(K_F)$ in an obvious way, we may regard $\okbd_{Y_\bullet}(K_Y), \okbd_{F_\bullet}(K_F) \subseteq  \R_{\geq 0}^{\dim X}$. Then ``$+$'' in (\ref{**}) is the Minkowski sum of convex bodies in $\R^{\dim X}$, and we will use this convention if no confusion is likely to occur.
%On the other hand, we may consider $\okbd_{Y_\bullet}(K_Y) \times %\okbd_{F_\bullet}(K_F)$ is a convex body in $\R_{\geq 0}^{\dim Y} \times \R_{\geq %0}^{\dim F} = \R_{\geq 0}^{\dim X}$ by regarding %$\Delta_{Y_\bullet}(K_Y)\subseteq\mathbb R^{dim Y}$ and %$\Delta_{F_\bullet}(K_F)\subseteq\mathbb R^{\dim F}$. Then we have
%$$
%\okbd_{Y_\bullet}(K_Y) + \okbd_{F_\bullet}(K_F) = \okbd_{Y_\bullet}(K_Y) \times %\okbd_{F_\bullet}(K_F).
%$$
As $K_X, K_Y, K_F$ are big by assumption, the Okounkov bodies
$$
\okbd_{X_\bullet}(K_X) \subseteq \R_{\geq 0}^{\dim X}, \okbd_{Y_\bullet}(K_Y) \subseteq \R_{\geq 0}^{\dim Y}, \okbd_{F_\bullet}(K_F) \subseteq \R_{\geq 0}^{\dim F}
$$
have full dimensions in the respective spaces. Thus their Euclidean volumes can be expressed as the volumes of divisors up to some constants depending on the dimensions of varieties (see \cite[Theorem A]{lm-nobody}). This leads to a new proof of Kawamata's product formula (\ref{*}) via (\ref{**}). It is worth recalling from \cite[Theorem 1.3]{CJPW} (see also \cite[Corollary 4.7]{T}) that the equality holds in (\ref{*}) (or equivalently, in (\ref{**})) if and only if $f$ is birationally isotrivial.

\medskip

The principal aim of this paper is to generalize the main results of \cite{CJPW}, and to provide a new approach toward the Iitaka conjecture. The following theorem illustrates how the Okounkov bodies can be used effectively in the study of Iitaka conjecture.

\begin{theorem}\label{thm:oklimK}
Let $f\colon X\to Y$ be an algebraic fiber space with general fiber $F$ over a point $\eta\in Y$. Assume that $K_Y$ and $K_F$ are pseudoeffective. Let $X_\bullet$ be a fiber-type admissible flag on $X$ associated to $Y_\bullet, F_\bullet$, where $Y_\bullet$ is an admissible flag on $Y$ centered at $\eta$ containing a positive volume subvariety $Y'$ of $K_Y$ and $F_\bullet$ is an admissible flag on $F$ containing a positive volume subvariety $F'$ of $K_F$. Then the following hold:

\smallskip

\noindent $(1)$ We have
$$
\oklim_{X_\bullet}(K_X) \supseteq \oklim_{Y_\bullet}(K_Y) + \oklim_{F_\bullet}(K_F).
$$
In particular,
$$
\nu_{\BDPP}(K_X) \geq \nu_{\BDPP}(K_Y) + \nu_{\BDPP}(K_F).
$$
\noindent $(2)$ Assume that $K_F$ is big.
If the equality $\nu_{\BDPP}(K_X) = \nu_{\BDPP}(K_Y) + \nu_{\BDPP}(K_F)$ holds, then the following canonical volume formula holds:
$$
\frac{\vol_{X|X'}^+(K_X)}{\nu_{\BDPP}(K_X)!}\geq\frac{\vol_{Y|Y'}^+(K_Y)}{\nu_{\BDPP}(K_Y)!}\cdot\frac{\vol_{F|F'}^+(K_F)}{\nu_{\BDPP}(K_F)!}
$$
where $X'=f^{-1}(Y')$. If the equality $\oklim_{X_\bullet}(K_X) = \oklim_{Y_\bullet}(K_Y) + \oklim_{F_\bullet}(K_F)$ holds, then $f$ is birationally isotrivial.

\end{theorem}

We refer to Section \ref{sec:prelim} for the definitions of positive volume subvariety, Nakayama subvariety, and the numerical Iitaka dimensions $\nu_{\BDPP}(D),  \kappa_{\sigma}(D)$ of a pseudoeffective divisor $D$.
The \emph{limiting Okounkov body} $\okbd_{X_\bullet}(D)$ of a pseudoeffective divisor $D$ was introduced and studied in \cite{CHPW-okbd I}. Recall that $\dim \oklim_{X_\bullet}(D) \leq \nu_{\BDPP}(D)$ in general.
If $X_\bullet$ contains a positive volume subvariety of $D$, then the equality holds and  the Euclidean volume of $\oklim_{X_\bullet} (D)$ in $\R^{\nu_{\BDPP}(D)}$ coincides with $\vol_{X|X'}^+(D)/\nu_{\BDPP}(D)!$.

\medskip

In contrast to \cite[Theorem 1.3]{CJPW}, the birational isotriviality of $f$ does not necessarily imply
the equality $\nu_{\BDPP}(K_X) = \nu_{\BDPP}(K_Y) + \nu_{\BDPP}(K_F)$ in Theorem \ref{thm:oklimK} (1). This means that the inclusion
$\oklim_{X_\bullet}(K_X) \supseteq \oklim_{Y_\bullet}(K_Y) + \oklim_{F_\bullet}(K_F)$ in Theorem \ref{thm:oklimK} (1) can be strict even when $f$ is birationally isotrivial and $K_F$ is big. See Example \ref{ex:isotrivialnotequal}.

\medskip

In the setting of Theorem \ref{thm:oklimK}, Nakayama \cite[Theorem V.4.1]{nakayama} and Fujino \cite[Theorem 2.1]{Fujino2} (see also the discussion in \cite[Section 3]{Fujino2}) proved the following inequalities
$$
\kappa_{\sigma}(K_X) \geq \kappa_{\sigma}(K_Y) + \kappa(K_F)~~\text{and}~~\kappa_{\sigma}(K_X) \geq \kappa(K_Y) + \kappa_{\sigma}(K_F).
$$
However, we still do not know whether the following inequality holds:
$$
\kappa_{\sigma}(K_X) \geq \kappa_{\sigma}(K_Y) + \kappa_{\sigma}(K_F).
$$
We remark that $\kappa_\sigma(D)\geq\nu_{\BDPP}(D)$ holds for every pseudoeffective divisor $D$ and the inequality can be strict for some $\R$-divisor $D$ (see \cite[Theorem 1.2]{CP}).
Notice that  $\nu_{\BDPP}(D)$ is the smallest one among the mostly used numerical Iitaka dimensions in the literature (see e.g., \cite{BDPP}, \cite{CP}, \cite{E}, \cite{lehmann-nu}, \cite{nakayama}) while $\nu_{\BDPP}(D)\geq \kappa(D)$. Thus the inequality in Theorem \ref{thm:oklimK} is in some sense closer to what is expected by the  Iitaka conjecture than the previously known results. Note also that through Theorem \ref{thm:oklimK}, we can also confirm that the \emph{abundance conjecture} \cite[Conjecture 3.8]{BDPP}, which predicts that $\nu_{\BDPP}(K_X)=\kappa(K_X)$ for every smooth projective variety $X$, implies the Iitaka conjecture. See Remark \ref{rem:anothernumIitaka} for another numerical variant of the Iitaka conjecture.

\medskip

The \emph{valuative Okounkov body} $\okval_{X_\bullet}(D)$ of an effective divisor $D$ is the Okounkov body constructed in the usual way without taking the limit process as we did for $\oklim_{X_\bullet}(D)$. It is introduced and studied in \cite{CHPW-okbd I}. It is known that $\dim \okval_{X_\bullet}(D) = \kappa(D)$ and if $X_\bullet$ contains a Nakayama subvariety $X'$ of $D$ and $X_{\dim X}$ is a general point, then the Euclidean volume of $\okval_{X_\bullet}(D)$ in $\R^{\kappa(D)}$ is $\vol_{X|X'}(D)/\kappa(D)!$. Theorem \ref{thm:oklimK} naturally leads us to ask the following intriguing question:

\medskip

\noindent \textbf{Question.}
\emph{Let $f\colon X\to Y$ be an algebraic fiber space with general fiber $F$ over a point $\eta\in Y$. Assume that $K_Y$ and $K_F$ are effective. Let $X_\bullet$ be a fiber-type admissible flag on $X$ associated to $Y_\bullet, F_\bullet$, where $Y_\bullet$ is an admissible flag on $Y$ centered at $\eta$ containing a Nakayama subvariety $Y'$ of $K_Y$ and $F_\bullet$ is an admissible flag on $F$ centered at a general point of $F$ containing a Nakayama subvariety $F'$ of $K_F$.
Then does the following inclusion
$$
\okval_{X_\bullet}(K_X) \supseteq \okval_{Y_\bullet}(K_Y) + \okval_{F_\bullet}(K_F)
$$
hold?}

\medskip

An affirmative answer to this question would solve the Iitaka conjecture because the dimensions of the valuative Okounkov bodies $\okval_{X_\bullet}(K_X), \okval_{Y_\bullet}(K_Y), \okval_{F_\bullet}(K_F)$ coincide with the Kodaira dimensions $\kappa(K_X),\kappa(K_Y),\kappa(K_F)$, respectively. However, this na\"{i}ve expectation fails to hold due to Example \ref{ex:inverseinclusion} for which the following reverse strict inclusion holds:
$$
\okval_{X_\bullet}(K_X) \subsetneq \okval_{Y_\bullet}(K_Y) + \okval_{F_\bullet}(K_F).
$$
We note that $K_F$ is big in this example. However, when $K_Y$ is big, we do have an affirmative answer to the question.

\begin{theorem}\label{thm:okvalK}
Let $f\colon X\to Y$ be an algebraic fiber space with general fiber $F$ over a point $\eta\in Y$. Assume that $K_Y$ is big and $K_F$ is effective. Let $X_\bullet$ be a fiber-type admissible flag on $X$ associated to $Y_\bullet, F_\bullet$, where $Y_\bullet$ is an admissible flag on $Y$ centered at $\eta$ containing a Nakayama subvariety $Y'$ of $K_Y$ and $F_\bullet$ is an admissible flag on $F$ centered at a general point of $F$ containing a Nakayama subvariety $F'$ of $K_F$. Then we have
$$
\okval_{X_\bullet}(K_X) \supseteq \okval_{Y_\bullet}(K_Y) + \okval_{F_\bullet}(K_F).
$$
In particular,
$$
\kappa(K_X) =\kappa(K_Y) + \kappa(K_F).
$$
\end{theorem}

Recall that the Iitaka conjecture for the case where $Y$ is of general type has been settled previously by Viehweg \cite[Corollary IV]{V1}. Note that the first statement in Theorem \ref{thm:okvalK} only implies that $\kappa(K_X) \geq \kappa(K_Y) + \kappa(K_F)$, but then the easy addition lemma shows the desired equality.
When $K_F$ is also big in Theorem \ref{thm:okvalK}, the similar statements to Theorem \ref{thm:oklimK} (2) hold, but this case is already treated in \cite{CJPW}.

\medskip

Both Theorems \ref{thm:oklimK} and \ref{thm:okvalK} easily follow from a more general statement about the subadditivity of Okounkov bodies of divisors on algebraic fiber spaces stated below. The following subadditivity theorem is a generalization of \cite[Theorem 1.1]{CJPW}.

\begin{theorem}\label{thm:main}
Let $f\colon X\to Y$ be an algebraic fiber space with general fiber $F$ over a point $\eta \in Y$. Let $D,R$ be divisors on $X$, and $D_Y, A_Y$ be divisors on $Y$ satisfying the following conditions:
\begin{enumerate}[$(1)$]
 \item $D \sim_{\Q} f^*D_Y + R$.
 \item $f_*\mathcal{O}_X(mR)$ is weakly positive for every sufficiently divisible integer $m > 0$.
 \item $R|_F$ is effective.
 \item $D_Y$ is effective.
 \item $A_Y$ is ample.
\end{enumerate}
Let $X_\bullet$ be a fiber-type admissible flag on $X$ associated to $Y_\bullet, F_\bullet$, where $Y_\bullet$ is an admissible flag on $Y$ centered at $\eta$ containing a Nakayama subvariety of $D_Y$, and $F_\bullet$ is an admissible flag on $F$ centered at a general point of $F$ containing a Nakayama subvariety of $R|_F$. Then we have
$$
\okval_{X_\bullet}(D+f^*A_Y) \supseteq \okval_{Y_\bullet}(D_Y) + \okval_{F_\bullet}(R|_F).
$$
\end{theorem}

As consequences of Theorem \ref{thm:main}, we obtain two Corollaries \ref{cor:oklim} and \ref{cor:okval}, which in turn imply Theorems \ref{thm:oklimK} and \ref{thm:okvalK} by the weak positivity theorem of Viehweg \cite[Theorem III]{V1} and its variants. See Section \ref{sec:prelim} for the definition of weakly positive sheaves.

\medskip

The rest of the paper is organized as follows. We begin in Section \ref{sec:prelim} with recalling basic definitions and facts. Section \ref{sec:main} is devoted to the proofs of the theorems stated in the introduction. In Section \ref{sec:examples}, we present some relevant examples, and we formulate a refinement of the question in the introduction.

%\subsection*{Acknowledgement}

%%%%%%%%%%%%%%%%%%%%%%%%%%%%%%%%%%%%%%%%%%%%%%%%%%%%%
\section{Preliminaries}\label{sec:prelim}
%%%%%%%%%%%%%%%%%%%%%%%%%%%%%%%%%%%%%%%%%%%%%%%%%%%%%

\subsection{Algebraic fiber spaces and weak positivity}
Let $f \colon X \to Y$ be an algebraic fiber space with general fiber $F$. Recall that the \emph{variation $\var(f)$ of $f$} is defined as the minimum of the transcendental degrees $\text{tr.}\deg_\C L$ of algebraically closed subfields $L\subseteq\overline{\C(Y)}$ such that $F\times_{\Spec (L)} \Spec (\overline{\C(Y)})$ is birationally equivalent to $X\times_{Y} \Spec (\overline{\C(Y)})$ for some smooth projective variety $F$ over $L$ (see \cite{V1}). We say that $f$ is \emph{birationally isotrivial} if  there exists a generically finite cover $\tau \colon Y' \to Y$ such that the fiber product $X \times_Y Y'$ is birationally equivalent to $F\times Y'$. It is well known that $f$ is birationally isotrivial if and only if $\var(f)=0$.

\begin{theorem}[{\cite[Theorem 1.1]{K1}, \cite[Theorem 1.20]{V2}; see also \cite[Theorem 3.4.7]{Fujino-book}}]\label{thm:var}
Let $f \colon X \to Y$ be an algebraic fiber space with general fiber $F$. Suppose that $F$ is a variety of general type. If $D_Y$ is any divisor on $Y$ with $\kappa(D_Y) \geq 0$, then we have
$$
\kappa(K_{X/Y} + f^*D_Y ) \geq \kappa(K_F) + \max\{ \kappa(D_Y), \var(f) \}.
$$
\end{theorem}

The \emph{Iitaka dimension} of a divisor $D$ on a smooth projective variety is defined as
$$
\kappa(D):= \max\left\{k\in\mathbb Z_{\geq0}\left|\limsup\limits_{m\to\infty}\frac{h^0(X,\mathcal O_X(\lfloor mD\rfloor))}{m^k}>0\right.\right\}
$$
if $h^0(X, \mathcal{O}_X(\lfloor mD \rfloor)) > 0$ for some integer $m>0$ and $\kappa(D):=-\infty$ otherwise.

\medskip

Let $\mathcal{F}$ be a coherent torsion-free sheaf on a smooth projective variety $X$.
We say that $\mathcal{F}$ is \emph{weakly positive} if there exists a nonempty open subvariety $U \subseteq X$ such that for every ample Cartier divisor $H$ on $X$ and every integer $m>0$,  the map
$$
H^0(X,\hat{S}^{mk}(\mathcal{F})\otimes \mathcal{O}_X(kH))\otimes \mathcal{O}_X\to \hat{S}^{mk}(\mathcal{F})\otimes \mathcal{O}_X(kH)
$$
is surjective at each point in $U$ for some integer $k > 0$, where $\hat{S}^{mk}(\mathcal{F}):=\big(S^{mk}(\mathcal{F})\big)^{\vee\vee}$ is the double dual of the sheaf $S^{mk}(\mathcal{F})$.

\begin{theorem}[{\cite[Theorem D]{DM}, \cite[Theorem 3.4.7]{Fujino-book}}]\label{thm:twistedweakpos}
Let $f \colon X \to Y$ be an algebraic fiber space with general fiber $F$, and $\Delta$ be an effective divisor on $X$ such that $(X, \Delta)$ is an lc pair and $K_X +\Delta$ is $\Q$-Cartier $\Q$-divisor. Then $f_*\mathcal{O}_X(k(K_{X/Y} + \Delta))$ is weakly positive for every positive integer $k$ such that $k(K_X+\Delta)$ is $\Q$-linearly equivalent to a Cartier divisor.
\end{theorem}

We refer to \cite{KM} for the basic definitions of the singularities of pairs. When $\Delta = 0$, Theorem \ref{thm:twistedweakpos} is nothing but Viehweg's weak positivity theorem \cite[Theorem III]{V1}.

\subsection{Okounkov bodies}
Let $X$ be a smooth projective variety of dimension $n$, and fix an \emph{admissible flag} $X_\bullet$ on $X$
$$
X_\bullet: X=X_0 \supseteq X_1\supseteq \cdots \supseteq X_{n} \supseteq X_{n} = \{x \}
$$
where each $X_i$ is an irreducible closed subvariety of $X$ having codimension $i$ and smooth at the point $x$.
Let $D$ be an effective divisor on $X$. We consider a valuation-like function
$$
\nu_{X_\bullet} \colon |D|_\Q \to \R^{n}_{\geq 0}, ~~D' \mapsto \nu_{X_\bullet}(D')=(\nu_1, \nu_2, \ldots, \nu_n)
$$
where the $\nu_i$ are defined inductively as follows:
\begin{enumerate}[$(1)$]
  \item let $\nu_1:=\mult_{X_1}D'$ and $D'_1:=D'-a_1X_1$, and then,
  \item assuming that we have defined $\nu_i$ and $D'_i$, define $\nu_{i+1}:=\mult_{X_{i+1}}(D'_i|_{X_i})$ and $D'_{i+1}=D'_i|_{X_i}-a_{i+1}X_{i+1}$.
\end{enumerate}
The \emph{Okounkov body} of $D$ with respect to $X_\bullet$ is defined as
$$
\Delta_{X_\bullet}(D):=\text{the convex closure of }\nu_{X_\bullet}(|D|_{\Q}) \text{ in } \R^{n}_{\geq 0}.
$$
By \cite[Proposition 3.3]{B}, we have $\dim \Delta_{X_\bullet}(D) = \kappa(D)$. When $D$ is not big, we use the notation $\okval_{\bullet}(D)$, which is called the \emph{valuative Okounkov body} of $D$ with respect to $X_\bullet$, for the Okounkov body $\Delta_{X_\bullet}(D)$. By \cite[Theorem A]{lm-nobody}, we have
$$
\vol_{\R^n}(\Delta_{X_\bullet}(D) = \frac{1}{n!}\vol_X(D)~~\text{for every admissible flag $X_\bullet$ on $X$}.
$$

\medskip

Recall that the \emph{restricted volume} of a divisor $D$ along an irreducible closed subvariety $V$ of dimension $v$ is defined as
$$
\vol_{X|V}(D):=\limsup_{m \to \infty} \frac{\dim \im \big( H^0(X, \mathcal{O}_X(\lfloor mD \rfloor ))\to H^0(V,\mathcal{O}_V(\lfloor mD \rfloor|_V))\big)}{m^v/v!}.
$$
When $V=X$, we simply set $\vol_X(D):=\vol_{X|X}(D)$, and we call it the \emph{volume} of $D$. The \emph{augmented restricted volume} of $D$ along $V$ is defined as
$$
\vol_{X|V}^+(D):=\lim_{\varepsilon \to 0+} \vol_{X|V}(D+\varepsilon A),
$$
where $A$ is an ample divisor on $X$. It is easy to check that the definition is independent of the choice of $A$ and $\vol_{X|V}^+(D)$ depends only on the numerical class of $D$. If $D$ is big,  then $\vol_{X|V}^+(D)=\vol_{X|V}(D)$. For more details on the restricted volumes, see \cite{CHPW-okbd I}, \cite{elmnp-restricted vol and base loci}.

\medskip

Now, let $D$ be a pseudoeffective divisor on $X$. The \emph{limiting Okounkov body} of $D$ with respect to $X_\bullet$  is defined as
$$
\oklim_{X_\bullet}(D):=\bigcap_{\varepsilon > 0} \Delta_{X_\bullet}(D+\varepsilon A)
$$
for any fixed ample divisor $A$ on $X$. This definition is independent of the choice of $A$. The limiting Okounkov body $\oklim_{X_\bullet}(D)$ depends only on the numerical class of $D$ (see \cite[Theorem C]{CHPW-okbd I}). If $D$ is big, then $\Delta_{X_\bullet}(D)=\okval_{X_\bullet}(D) = \oklim_{X_\bullet}(D)$. By \cite[Theorem 1.1]{CP}, we have $\kappa(D) \leq \dim \oklim_{X_\bullet}(D) \leq \nu_{\BDPP}(D)$ for every admissible flag $X_\bullet$.

\medskip

The following numerical Iitaka dimension was introduced by Boucksom--Demailly--P\u{a}un--Peternell \cite{BDPP}:
$$
\nu_{\BDPP}(D):=\max\left\{k\in\mathbb Z_{\geq 0}\left|\langle D^k\rangle\neq 0\right.\right\}
$$
where $\langle D^k \rangle$ is the positive intersection product (see \cite[Section 4]{lehmann-nu} for the definition and basic properties). 
%The dimension  $\nu_{\BDPP}(D)$  for a nef divisor $D$ was first defined by Kawamata. 
By \cite[Theorem 6.2]{lehmann-nu} (see also \cite[Theorem 1.1]{CP}), we have
$$
\nu_{\BDPP}(D) = \max\left\{ \dim W \left| \vol^+_{X|W}(L) > 0\right. \right\}
$$
where the $W$ range over all irreducible closed subvarieties of $X$. We now recall some other numerical Iitaka dimensions introduced by Nakayama \cite{nakayama} and Lehmann \cite{lehmann-nu}:
$$
\arraycolsep=1.4pt\def\arraystretch{1.9}
\begin{array}{rl}
\kappa_{\sigma}(D)&:=\displaystyle  \max \left\{ k \in \Z_{\geq 0} \left| \limsup\limits_{m \to \infty} \frac{h^0(X, \lfloor mD \rfloor + A)}{m^k} > 0 \right.\right\} \\
\kappa_{\vol}(D)&:=\displaystyle \max \left\{k\in \Z_{\geq 0}\left| \liminf\limits_{\varepsilon \to 0} \frac{\vol_X(D+ \varepsilon A)}{\varepsilon^{n-k}} > 0\right. \right\}
\end{array}
$$
where $A$ is a sufficiently positive ample $\Z$-divisor on $X$. It is well known that the numerical Iitaka dimensions $\nu_{\BDPP}(D), \kappa_{\sigma}(D), \kappa_{\vol}(D)$ depend only on the numerical class of $D$. Furthermore, $\nu_{\BDPP}(D), \kappa_{\sigma}(D), \kappa_{\vol}(D)$ are nonnegative integers at most $n$ when $D$ is pseudoeffective, and $ \nu_{\BDPP}(D), \kappa_{\sigma}(D), \kappa_{\vol}(D) =n$ if and only if $D$ is big. By \cite[Proposition 3.1]{CP}, $\nu_{\BDPP}(D) \leq \kappa_{\sigma}(D), \kappa_{\vol}(D)$, and by \cite[Theorem 1.2]{CP}, the inequality can be strict (see also \cite{lesieutre}). We refer to \cite{CP} for more basic properties of numerical Iitaka dimensions.

\medskip

For an effective divisor $D$, an irreducible closed subvariety $U \subseteq X$ is called a \emph{Nakayama subvariety} of $D$ if $\dim U = \kappa(D)$ and the natural restriction map
$$
H^0(X, \mathcal{O}_X(\lfloor mD \rfloor)) \to H^0(U, \mathcal{O}_U(\lfloor mD \rfloor|_U))
$$
is injective for every integer $m \geq 0$. By construction, the restriction $D|_U$ is big on $U$. For a pseudoeffective divisor $D$, an irreducible closed subvariety $V \subseteq X$ is called a \emph{positive volume subvariety} of $D$ if $\dim V = \nu_{\BDPP}(D)$ and $\vol_{X|V}^+(D) > 0$.
These subvarieties were first defined in \cite{CHPW-okbd I}.

\begin{theorem}[{\cite[Theorems A and B]{CHPW-okbd I}}]\label{CHPWmain}
Let $X$ be a smooth projective variety of dimension $n$, and $D$ be a divisor on $X$.
\begin{enumerate}[leftmargin=0cm,itemindent=.6cm]
\item[$(1)$] Suppose that $D$ is effective. Fix an admissible flag $X_\bullet$ containing a Nakayama subvariety $U$ of $D$ such that $Y_n$ is a general point in $X$. Then $\okval_{X_\bullet}(D) \subseteq \{0 \}^{n-\kappa(D)} \times \R^{\kappa(D)}$ so that one can regard $\okval_{X_\bullet}(D) \subseteq \R^{\kappa(D)}$. Furthermore, we have
$$\dim \okval_{X_\bullet}(D)=\kappa(D) \text{ and } \vol_{\R^{\kappa(D)}}(\okval_{X_\bullet}(D))=\frac{1}{\kappa(D)!} \vol_{X|U}(D).$$
\item[$(2)$] Suppose that $D$ is pseudoeffective. Fix an admissible flag $X_\bullet$ containing a positive volume subvariety $V$ of $D$. Then $\oklim_{X_\bullet}(D) \subseteq \{0 \}^{n-\nu_{\BDPP}(D)} \times \R^{\nu_{\BDPP}(D)}$ so that one can regard $\oklim_{X_\bullet}(D) \subseteq \R^{\nu_{\BDPP}(D)}$. Furthermore, we have
$$\dim \oklim_{X_\bullet}(D)=\nu_{\BDPP}(D) \text{ and } \vol_{\R^{\nu_{\BDPP}(D)}}(\oklim_{X_\bullet}(D))=\frac{1}{\nu_{\BDPP}(D)!} \vol_{X|V}^+(D).$$
\end{enumerate}
\end{theorem}

We refer to \cite{B}, \cite{CHPW-okbd I, CPW-okbd II}, \cite{KK}, \cite{lm-nobody} for more basic properties of Okounkov bodies. See also \cite[Remark 3.5]{CP} and \cite[Remark 2.6]{CPW-okbdab}.

%%%%%%%%%%%%%%%%%%%%%%%%%%%%%%%%%%%%%%%%%%%%%%%%%%%%%
\section{Proofs of main results}\label{sec:main}
%%%%%%%%%%%%%%%%%%%%%%%%%%%%%%%%%%%%%%%%%%%%%%%%%%%%%

In this section, we prove the main result of this paper, Theorem \ref{thm:main}, and its consequences, Theorems \ref{thm:oklimK} and \ref{thm:okvalK}. The following lemma is the key ingredient of the proof.

\begin{lemma}\label{lem:maintechnical}
Let $f \colon X \to Y$ be a surjective morphism between smooth projective varieties $X$ and $Y$ with connected fibers, and $F$ be a general fiber of $f$. Let $D, R$ be divisors on $X$, and $D_Y$ be a divisor on $Y$ satisfying the following conditions:
\begin{enumerate}
 \item[$(1)$] $D \sim_{\Q} f^*D_Y + R$.
 \item[$(2)$] $f_*\mathcal{O}_X(mR)$ is weakly positive for every sufficiently divisible integer $m > 0$.
 \item[$(3)$] $R|_F$ is effective.
 \item[$(4)$] $D_Y$ is big.
\end{enumerate}
Let $N \subseteq F$ be a Nakayama subvariety of $R|_F$. Then we have
$$
\vol_{X|N}(D) = \vol_{F|N}(R|_F ).
$$
In particular, $\kappa(D)\geq\kappa(R|_F)$ holds.
\end{lemma}

\begin{proof}
The proof is similar to the proof of \cite[Lemma 3.1]{CJPW}, but we present the full details for readers' convenience.
As $R|_F = D|_F$, we have $\vol_{F|N}(R|_F)=\vol_{F|N}(D|_F)$. This implies that
$$
\vol_{X|N}(D) \leq \vol_{F|N}(R|_F).
$$
For an ample divisor $H$ on $Y$, notice that
$$
\vol_{X|N}(D) \geq \vol_{X|N}\left(R+ \frac{1}{m}f^*H\right)
$$
for any sufficiently large integer $m>0$. Thus it is enough to show that
\begin{equation}\label{eq:keylemeq1}
\vol_{X|N}\left(R+ \frac{1}{m}f^*H\right) \geq \vol_{F|N}(R|_F)
\end{equation}
for any sufficiently large and divisible integer $m>0$.

Note that $\vol_{F|N}(R|_F) > 0$ because $N$ is a Nakayama subvariety of $R|_F$.
Let $W_\bullet$ be a graded linear series on $N$ associated to $R|_N$ consisting of subspaces
$$
W_i := \im \big( H^0(F, \lfloor iR |_F \rfloor) \to H^0(N, \lfloor iR|_F \rfloor|_N) \big) \subseteq H^0(N, \lfloor iR|_N \rfloor)
$$
for each integer $i \geq 0$. Then $\vol_{N}(W_\bullet) =\vol_{F|N}(R|_F) > 0$.
By Fujita approximation \cite[Theorem 3.14]{DP}, for any $\varepsilon > 0$, we have
\begin{equation}\label{eq:keylemeq2}
\frac{\dim \im \big( S^{p_1} W_{p_2} \to W_{p_1p_2} \big)}{(p_1p_2)^n/n!} \geq \vol_{F|N}(R|_F) - \varepsilon
\end{equation}
for every sufficiently large and divisible integers $p_1,p_2>0$, where $n=\dim N$.

Now, by \cite[Lemma 7.3]{V1}, we have a commutative diagram
$$
\xymatrix{
X' \ar[d]_-{f'} \ar[r]^-{\tau'} & X \ar[d]^-{f} \\
Y' \ar[r]_-{\tau} & Y
}
$$
such that $\tau \colon Y' \to Y$ is a birational morphism with $Y'$ smooth projective, $X'$ is a resolution of singularities of the main component of $X \times_Y Y'$, and $f' \colon X' \to Y'$ is the induced morphism such that every divisor $B'$ on $X'$ with $\codim f'(B') \geq 2$ is $\tau'$-exceptional divisor.
We may assume that $\tau$ is isomorphic over a neighborhood of $f(F)$, so that we may regard that $F$ is also a general fiber of $f'$ and $R'|_F=R|_F$, where $R':=\tau'^*R$.
For a sufficiently large and divisible integer $m_1>0$, there is a map
$$
\tau^*f_*\mathcal{O}_{X}(m_1R) \to f'_*\mathcal{O}_{X'}(m_1R'),
$$
which is surjective over some open subset of $Y'$. Thus
$f'_*\mathcal{O}_{X'}(m_1R')$ is weakly positive on $Y'$. Let $H$ be an ample Cartier divisor on $Y$, and $H':=\tau^*H$. Then there is some integer $k>0$ such that $\hat{S}^k f'_*\mathcal{O}_{X'}(m_1R') \otimes \mathcal{O}_{Y'}(kH')$ is generated by global sections over some open subset of $Y'$.

For any integer $m_0 > 0$, consider the map
$$
\varphi_{m_0} \colon \hat{S}^{m_0}\big(\hat{S}^k\big(f'_*\mathcal{O}_{X'}(m_1R') \big) \otimes \mathcal{O}_{Y'}(kH')\big) \to \big( f'_*\mathcal{O}_{X'}(m_0m_1 kR') \otimes \mathcal{O}_{Y'}(m_0kH') \big)^{\vee\vee}.
$$
By \cite[Lemma 2.7]{CJPW}, we can find an effective divisor $B$ on $X'$ such that $\codim f'(B) \geq 2$ and
$$
\big(f'_*\mathcal{O}_{X'}\big(m_0m_1 kR' + f'^*(m_0 k H') \big) \big)^{\vee\vee} = f'_*\mathcal{O}_{X'}\big(m_0m_1 kR' + f'^*(m_0 k H') +B \big).
$$
Note that $B$ is $\tau'$-exceptional. Thus we have
$$
H^0\big(Y', \big( f'_* \mathcal{O}_{X'}\big(m_0m_1 kR' + f'^*(m_0 k H') \big) \big)^{\vee\vee} \big)= H^0\big(X, \mathcal{O}_X \big( m_0m_1 kR + f^*(m_0kH) \big) \big).
$$
For any sufficiently large and divisible integers $m_0, m_1 > 0$, consider the commutative diagram
$$
\xymatrix{
S^{m_0}S^k H^0(F, \mathcal{O}_F(m_1 R|_F)) \ar[r]^-{\psi} & H^0(N, \mathcal{O}_N(m_0m_1kR|_N)) \\
H^0(Y', \hat{S}^{m_0}(\hat{S}^k(f_*'\mathcal{O}_{X'}(m_1R')) \otimes \mathcal{O}_{Y'}(kH')) \ar[r]^-{H^0(\varphi_{m_0})} \ar[u] & H^0(X, \mathcal{O}_X(m_0m_1kR + f^*(m_0kH))). \ar[u]
}
$$
By the generic global generation of $\hat{S}^k(f_*'(\mathcal{O}_{X'}(m_1R')) \otimes \mathcal{O}_{Y'}(kH'))$, the vertical upward map on the left is surjective. By Fujita approximation (\ref{eq:keylemeq2}), we obtain
$$
\frac{\dim \im(\psi)}{(m_0m_1k)^n/n!} \geq \vol_{F|N}(R|_F) - \varepsilon
$$
for a sufficiently small number $\varepsilon > 0$. Thus we obtain
\begin{small}
\begin{equation}\label{eq:keylemeq3}
\frac{\dim \im\big(H^0(X, m_0m_1kR + f^*(m_0kH)) \to H^0(N, m_0m_1kR|_N) \big)}{(m_0m_1k)^n/n!}\geq \vol_{F|N}(R|_F) - \varepsilon.
\end{equation}
\end{small}\\[-7pt]
Here we can make $\varepsilon$ arbitrarily small by taking larger integers $m_0, m_1>0 $.
Notice that
\begin{footnotesize}
$$
\vol_{X|N}\left(R + \frac{1}{m_1}f^*H \right) =  \limsup_{m_0 \to \infty} \frac{\dim \im\big(H^0(X, m_0m_1kR + f^*(m_0kH)) \to H^0(N, m_0m_1kR|_N) \big)}{(m_0m_1k)^n/n!}.
$$
\end{footnotesize}\\[-10pt]
Hence (\ref{eq:keylemeq3}) implies (\ref{eq:keylemeq1}), and this completes the proof.
\end{proof}

Now, we are ready to prove Theorem \ref{thm:main}.

\begin{proof}[Proof of Theorem \ref{thm:main}]
First, we write
$$
D + f^*A_Y \sim_{\Q} f^*D_Y + (R+f^*A_Y).
$$
By the usual subadditivity property of Okounkov bodies, we have
\begin{equation}\label{eq:okbdsubadditivity}
\okval_{X_\bullet}(D+f^*A_Y) \supseteq \okval_{X_\bullet}(f^*D_Y) + \okval_{X_\bullet}(R+f^*A_Y).
\end{equation}
It is clear that
$$
\okval_{X_\bullet}(f^*D_Y)=\okval_{Y_\bullet}(D_Y)\times\{0\}^{\dim F} \subseteq \R_{\geq 0}^{\dim Y}\times\{0\}^{\dim F}.\\
$$
Thus we may identify $\okval_{Y_\bullet}(D_Y)$ with $\okval_{X_\bullet}(f^*D_Y)$.
We can also regard
$$
\okval_{F_\bullet}(R|_F) \subseteq \{ 0\}^{\dim Y} \times \R_{\geq 0}^{\dim F}.
$$
To prove the theorem, by considering (\ref{eq:okbdsubadditivity}), it is sufficient to show that $\okval_{F_\bullet}(R|_F)$ is a subset of $ \okval_{X_\bullet}(R+f^*A_Y)$.
To this end, let $N$ be the Nakayama subvariety of $R|_F$ contained in the admissible flag $F_\bullet$ so that $\dim N = \kappa(R|_F)$.
By Theorem \ref{CHPWmain}, we have
\begin{equation}\label{eq:okvalRF}
\okval_{F_\bullet}(R|_F) \subseteq \R^{\kappa(R|_F)} \text{ and }
\kappa(R|_F)!\cdot  \vol_{\R^{\kappa(R|_F)}}\okval_{F_\bullet}(R|_F) =\vol_{F|N}(R|_F).
\end{equation}
Now, consider the graded linear series $W_\bullet$ on $N$ associated to $(R+f^*A_Y)|_N = R|_N$ consisting of subspaces
$$
W_i = \text{Image}\big(H^0(X, \lfloor iR+if^*A_Y \rfloor) \to H^0(N, \lfloor iR+if^*A_Y \rfloor|_N) \big)
$$
for each integer $i \geq 0$. Then Lemma \ref{lem:maintechnical} and (\ref{eq:okvalRF}) imply that
$$
\vol_N(W_\bullet) = \vol_{X|N}(R+f^*A_Y) = \vol_{F|N}(R|_F) = \kappa(R|_F)!\cdot  \vol_{\R^{\kappa(R|_F)}}\okval_{F_\bullet}(R|_F).
$$
Let $N_\bullet$ be an admissible flag on $N$ such that $N_i = X_{i + \dim X - \dim N}$ for all $0 \leq i \leq \dim N$.
Since $N_{\dim N}=X_{\dim X}$ is a general point, it follows from \cite[Remark 2.8 and Theorem 2.13]{lm-nobody} that
$$
\vol_{\R^{\kappa(R|_F)}}(\Delta_{N_\bullet}(W_\bullet)) = \frac{1}{\kappa(R|_F)!} \vol_N(W_\bullet) =  \vol_{\R^{\kappa(R|_F)}}\okval_{F_\bullet}(R|_F).
$$
As we have
\begin{equation*}\label{eq:mainthmeq1}
\Delta_{N_\bullet}(W_\bullet) \subseteq \okval_{X_\bullet}(R+f^*A_Y)_{x_1=\cdots=x_{\dim Y}=0} \subseteq  \okval_{F_\bullet}(R|_F),
\end{equation*}
we see that $\okval_{X_\bullet}(R+f^*A_Y)_{x_1=\cdots=x_{\dim Y}=0} = \okval_{F_\bullet}(R|_F)$.
Hence we may regard $\okval_{F_\bullet}(R|_F)$ as a subset of $ \okval_{X_\bullet}(R+f^*A_Y)$, and we complete the proof.
\end{proof}

\begin{remark}\label{rem:genptassump}
The assumption that $F_\bullet$ is centered at a general point of $F$ in Theorem \ref{thm:main} is only used when we apply Theorem \ref{CHPWmain} and \cite[Remark 2.8]{lm-nobody}. Suppose that $R|_F$ is big. Then $F$ is the unique Nakayama subvariety of $R|_F$, and $W_\bullet$ satisfies Condition
(B) in \cite[Definition 2.5]{lm-nobody}. In this case, we do not need the assumption that $F_\bullet$ is centered at a general point of $F$ because we can apply \cite[Theorem A]{lm-nobody} instead and we do not need \cite[Remark 2.8]{lm-nobody}.
\end{remark}

\begin{corollary}\label{cor:oklim}
Let $f\colon X\to Y$ be an algebraic fiber space with general fiber $F$. Let $D,R$ be divisors on $X$, and $D_Y$ be a divisor on $Y$ satisfying the following conditions:
\begin{enumerate}
 \item[$(1)$] $D \sim_{\Q} f^*D_Y + R$.
 \item[$(2)$] There exists an ample divisor $A$ on $X$ such that  for any sufficiently small rational number $\varepsilon > 0$, the sheaf $f_*\mathcal{O}_X(m(R+\varepsilon A))$ is weakly positive for every sufficiently divisible integer $m > 0$.
  \item[$(3)$] $R|_F$ is pseudoeffective.
 \item[$(4)$] $D_Y$ is pseudoeffective.
\end{enumerate}
Let $X_\bullet$ be a fiber-type admissible flag on $X$ associated to $Y_\bullet, F_\bullet$, where $Y_\bullet$ is an admissible flag on $Y$ centered at the general point $f(F)$ of $Y$ containing a positive volume subvariety of $D_Y$, and $F_\bullet$ is an admissible flag on $F$ containing a positive volume subvariety of $R|_F$. Then we have
$$
\oklim_{X_\bullet}(D) \supseteq \oklim_{Y_\bullet}(D_Y) + \oklim_{F_\bullet}(R|_F).
$$
In particular,
$$
\nu_{\BDPP}(D) \geq \nu_{\BDPP}(D_Y) + \nu_{\BDPP}(R|_F).
$$
\end{corollary}

\begin{proof}
Let $A_Y$ be an ample divisor on $Y$. Then $D_Y+\delta A_Y$ is big for any $\delta > 0$, and $Y$ itself is a Nakayama subvariety of $D_Y+\delta A_Y$. Similarly, $F$ is a Nakayama subvariety of $(R+\varepsilon A)|_F$ for any $\varepsilon > 0$ as $(R+\varepsilon A)|_F$ is big. \cite[Lemma 2.4]{CJPW} implies that $D+\varepsilon A + \delta f^*A_Y$ is big for any $\varepsilon, \delta > 0$.
Then by Theorem \ref{thm:main} and Remark \ref{rem:genptassump},
\begin{equation}\label{eq:compareokbd+smallample}
\okbd_{X_\bullet}(D+\varepsilon A + 2\delta f^*A_Y) \supseteq \okbd_{Y_\bullet}(D_Y+\delta A_Y) +\okbd_{F_\bullet}((R+\varepsilon A)|_F)
\end{equation}
for any sufficiently small rational numbers $\varepsilon, \delta > 0$.
By taking $\varepsilon \to 0, \delta \to 0$, we obtain
$$
\oklim_{X_\bullet}(D) \supseteq \oklim_{Y_\bullet}(D_Y) + \oklim_{F_\bullet}(R|_F).
$$
%Notice that $\oklim_{Y_\bullet}(D_Y) + \oklim_{F_\bullet}(R|_F) = %\oklim_{Y_\bullet}(D_Y) \times \oklim_{F_\bullet}(R|_F)$.
Now, by Theorem \ref{CHPWmain}, we have
\begin{align*}
\nu_{\BDPP}(D) \geq \dim \oklim_{X_\bullet}(D) &\geq \dim (\oklim_{Y_\bullet}(D_Y) + \oklim_{F_\bullet}(R|_F))\\
&=\dim \oklim_{Y_\bullet}(D_Y)+\dim\oklim_{F_\bullet}(R|_F)\\
& = \nu_{\BDPP}(D_Y) + \nu_{\BDPP}(R|_F),
\end{align*}
and we finish the proof.
\end{proof}

Theorem \ref{thm:oklimK} easily follows from Corollary \ref{cor:oklim}.

\begin{proof}[Proof of Theorem \ref{thm:oklimK}]
Let $A$ be any effective ample divisor on $X$. Then $(X, \varepsilon A)$ is a klt pair for any sufficiently small rational number $\varepsilon > 0$, and hence, $f_*\mathcal{O}_X(m(K_{X/Y} + \varepsilon A))$ is weakly positive for every sufficiently divisible integer $m > 0$ by Theorem \ref{thm:twistedweakpos}. Note that $K_X|_F = K_F$. Then the first part (1) of the theorem follows from Corollary \ref{cor:oklim}.

To prove the second part (2) of the theorem, assume that $K_F$ is big and the equality $\nu_{\BDPP}(K_X)=\nu_{\BDPP}(K_Y) + \nu_{\BDPP}(K_F)$ holds. As $\nu_{\BDPP}(K_F) = \dim F = \dim X - \dim Y$, we get $\dim Y - \nu_{\BDPP}(K_Y) = \dim X - \nu_{\BDPP}(X)$.
By Theorem \ref{CHPWmain}, we obtain the following inclusions
$$
\def\arraystretch{1.3}
\begin{array}{l}
\oklim_{F_\bullet}(K_F) \subseteq \{0\}^{\dim X - \nu_{\BDPP}(K_X)} \times \{0\}^{\nu_{\BDPP}(K_Y)}\times \R^{\nu_{\BDPP}(K_F)}\;\;\text{and}\\
\oklim_{Y_\bullet}(K_Y) \subseteq \{0\}^{\dim X - \nu_{\BDPP}(K_X)} \times \R^{\nu_{\BDPP}(K_Y)}\times \{0\}^{\nu_{\BDPP}(K_F)}
\end{array}
$$
which imply
$$
\def\arraystretch{1.3}
\begin{array}{l}
\oklim_{Y_\bullet}(K_Y)  + \oklim_{F_\bullet}(K_F) \subseteq \{0\}^{\dim X - \nu_{\BDPP}(K_X)} \times \R^{\nu_{\BDPP}(K_X)}.
\end{array}
$$
Furthermore, Theorem \ref{CHPWmain} also implies
$$
\def\arraystretch{1.3}
\begin{array}{l}
\dim (\oklim_{Y_\bullet}(K_Y)  + \oklim_{F_\bullet}(K_F))  = \nu_{\BDPP}(K_Y) + \nu_{\BDPP}(K_F) = \nu_{\BDPP}(K_X).
\end{array}
$$
Since $\nu_{\BDPP}(K_X) \geq \dim \oklim_{X_\bullet}(K_X) $ and $\oklim_{X_\bullet}(K_X) \supseteq \oklim_{Y_\bullet}(K_Y) + \oklim_{F_\bullet}(K_F)$, it follows that
$$
\oklim_{X_\bullet}(K_X) \subseteq \{0\}^{\dim X - \nu_{\BDPP}(K_X)} \times \R^{\nu_{\BDPP}(K_X)}
$$
and $\oklim_{X_\bullet}(K_X)$ has full dimension in $\R^{\nu_{\BDPP}(K_X)}$.
By \cite[Theorem 1.2]{CPW-okbd II}, $X_\bullet$ contains a positive volume subvariety of $X$. Then the canonical volume formula in (2) follows from Theorem \ref{CHPWmain}.

Finally, suppose that $K_F$ is big and $f$ is not birationally isotrivial so that $\var(f) > 0$. By Theorem \ref{thm:var}, we have
$$
\kappa(K_{X/Y}) \geq \kappa(K_F) + \var(f) > \kappa(K_F).
$$
Thus we get
$$
\dim \oklim_{X_\bullet}(K_{X/Y}) \geq \kappa(K_{X/Y}) > \kappa(K_F) = \dim \oklim_{F_\bullet}(K_F).
$$
Note that we have shown $\oklim_{X_\bullet}(K_{X/Y}) \supseteq \oklim_{F_\bullet}(K_F)$ in the proofs of Theorem \ref{thm:main} and Corollary \ref{cor:oklim}. The strict inequality above implies
$$
\oklim_{X_\bullet}(K_{X/Y}) \supsetneq \oklim_{F_\bullet}(K_F).
$$
Therefore, using the usual subadditivity of Okounkov bodies, we obtain
\begin{align*}
\oklim_{X_\bullet}(K_X) &\supseteq \oklim_{Y_\bullet}(K_Y) + \oklim_{X_\bullet}(K_{X/Y})\\
&\supsetneq \oklim_{Y_\bullet}(K_Y) + \oklim_{F_\bullet}(K_F).
\end{align*}
%This implies that
%$$
%\oklim_{X_\bullet}(K_X) \supsetneq \oklim_{Y_\bullet}(K_Y) + \oklim_{F_\bullet}(K_F).
%$$
This completes the proof.
\end{proof}

\begin{remark}\label{rem:anothernumIitaka}
In the setting of Corollary \ref{cor:oklim}, if we take an ample divisor $A_Y$ on $Y$ with $A \geq f^*A_Y$, then
(\ref{eq:compareokbd+smallample}) and \cite[Theorem A]{lm-nobody} show that
$$
\vol_X(D+3\varepsilon A) \geq \vol_X(D+\varepsilon A + 2\varepsilon f^*A_Y) \geq \vol_Y(D_Y + \varepsilon A_Y) \cdot \vol_F((R+\epsilon A)|_F)
$$
for any sufficiently small rational number $\varepsilon > 0$. This yields that
$$
\kappa_{\vol}(D) \geq \kappa_{\vol}(D_Y) + \kappa_{\vol}(R|_F).
$$
In particular, for any algebraic fiber space $f \colon X \to Y$ with general fiber $F$, we have
$$
\kappa_{\vol}(K_X) \geq \kappa_{\vol}(K_Y) + \kappa_{\vol}(K_F).
$$
\end{remark}

\begin{corollary}\label{cor:okval}
Let $f \colon X \to Y$ be an algebraic fiber space with general fiber $F$. Let $D,R$ be divisors on $X$, and $D_Y$ be a divisor on $Y$ satisfying the following conditions:
\begin{enumerate}
 \item[$(1)$] $D \sim_{\Q} f^*D_Y + R$.
 \item[$(2)$] $f_*\mathcal{O}_X(mR)$ is weakly positive for every sufficiently divisible integer $m > 0$.
 \item[$(3)$] $R|_F$ is effective.
 \item[$(4)$] $D_Y$ is big.
\end{enumerate}
Let $X_\bullet$ be a fiber-type admissible flag on $X$ associated to $Y_\bullet, F_\bullet$, where $Y_\bullet$ is an admissible flag on $Y$ centered at the general point $f(F)$ of $Y$ containing a Nakayama subvariety of $D_Y$, and $F_\bullet$ is an admissible flag on $F$ centered at a general point of $F$ containing a Nakayama of $R|_F$. Then we have
$$
\okval_{X_\bullet}(D) \supseteq \okval_{Y_\bullet}(D_Y) + \okval_{F_\bullet}(R|_F).
$$
In particular,
$$
\kappa(D) \geq \kappa(D_Y) + \kappa(R|_F).
$$
\end{corollary}

\begin{proof}
Let $A_Y$ be an ample divisor on $Y$. Then $D_Y - \varepsilon A_Y$ is big for any sufficiently small rational number $\varepsilon > 0$, so $Y$ itself is a Nakayama subvariety of $D_Y - \varepsilon A_Y$. We write
$$
D-f^* \varepsilon A_Y  \sim_{\Q} f^*(D_Y - \varepsilon A_Y) + R.
$$
By applying Theorem \ref{thm:main}, we see that
$$
\okval_{X_\bullet}(D) \supseteq \okval_{Y_\bullet}(D_Y - \varepsilon A_Y) + \okval_{F_\bullet}(R|_F)
$$
for any sufficiently small rational number $\varepsilon > 0$. By taking $\varepsilon \to 0$, we obtain
$$
\okval_{X_\bullet}(D) \supseteq \okval_{Y_\bullet}(D_Y) + \okval_{F_\bullet}(R|_F).
$$
We then have
$$
\kappa(D) = \dim \okval_{X_\bullet}(D) \geq \dim ( \okval_{Y_\bullet}(D_Y) + \okval_{F_\bullet}(R|_F)) = \kappa(D_Y) + \kappa(R|_F),
$$
and we finish the proof.
\end{proof}

Theorem \ref{thm:okvalK} is an immediate consequence of Corollary \ref{cor:okval}.

\begin{proof}[Proof of Theorem \ref{thm:okvalK}]
By Viehweg's weak positivity theorem \cite[Theorem III]{V1}, we know that $f_*\mathcal{O}_X(mK_{X/Y} )$ is weakly positive for every integer $m>0$. Note that $K_X|_F = K_F$.  Then the first statement of the theorem  follows from Corollary \ref{cor:okval}. We also have $\kappa(K_X) \geq \kappa(K_Y) + \kappa(K_F)$. Recall that $K_Y$ is big. The easy addition lemma shows that $\kappa(K_X)\leq\dim Y+\kappa(K_F) = \kappa(K_Y) + \kappa(K_F)$. Thus $\kappa(K_X)=\kappa(K_Y)+\kappa(K_F)$.
\end{proof}

%%%%%%%%%%%%%%%%%%%%%%%%%%%%%%%%%%%%%%%%%%%%%%%%%%%%%
\section{Examples}\label{sec:examples}
%%%%%%%%%%%%%%%%%%%%%%%%%%%%%%%%%%%%%%%%%%%%%%%%%%%%%

In this section, we exhibit some relevant examples. First, we give a counterexample to the converse of Theorem \ref{thm:oklimK} (2).

\begin{example}\label{ex:isotrivialnotequal}
Let $S$ be a smooth projective minimal surface of general type. Suppose that there is a genus two fibration $f \colon S \to C$ over an elliptic curve $C$ such that $f$ is birationally isotrivial. Let $F$ be a general fiber of $f$. Then $\nu_{\BDPP}(K_C)=\dim \oklim_{C_\bullet}(K_C) = 0$ for every admissible flag $C_\bullet$ on $C$ and $\nu_{\BDPP}(K_F)=\dim \oklim_{F_\bullet}(K_F)=1$ for every admissible flag on $F$. However, $\nu_{\BDPP}(K_S)=\dim \oklim_{S_\bullet}(K_S) = \dim \okbd_{S_\bullet}(K_S)=2$ for every admissible flag $S_\bullet$ on $S$. Therefore, we have
$$
\oklim_{S_\bullet}(K_S) \supsetneq \oklim_{C_\bullet}(K_C) + \oklim_{F_\bullet}(K_F)
$$
for every fiber-type admissible flag $S_\bullet$ associated to admissible flags $C_\bullet, F_\bullet$ on $C, F$, respectively. Note also that the strict inequality
$$
\nu_{\BDPP}(K_S)>\nu_{\BDPP}(K_C)+\nu_{\BDPP}(K_F)
$$
holds even though $f$ is birationally isotrivial.
\end{example}

Next, we give a counterexample to the question stated in the introduction.

\begin{example}\label{ex:inverseinclusion}
Let $S$ be a smooth projective minimal surface with $\kappa(K_S)=1$. Suppose that $q(S)=2$ and there is a smooth genus two fibration $f \colon S \to C$ over an elliptic curve $C$. Then $f$ is birationally isotrivial, and the Iitaka fibration $g \colon S \to C'$ of $S$ has two double fibers (see \cite[Lemma 2.1]{Kar}). Notice that $S$ is not biratioinal to $C \times F$, where $F$ is a general fiber of $f$.
Let $x$ be a general point on $F$, and consider the admissible flag on $S$:
$$
S_\bullet: S \supseteq F \supseteq \{x\},
$$
which is a fiber-type admissible flag associated to $C_{\bullet}: C \supseteq \{f(F)\}$ and $F_\bullet: F \supseteq \{x\}$. Note that $\okval_{C_\bullet}(K_C)=\{0\}$. We have
$$
\okval_{S_\bullet}(K_S) = \frac{1}{\deg g|_F} \oklim_{S_\bullet}(K_S) = \frac{1}{\deg g|_F} \okval_{F_\bullet}(K_F).
$$
As $\deg g|_F \geq 2$, we get
$$
\okval_{S_\bullet}(K_S) \subsetneq  \okval_{C_\bullet}(K_C) + \okval_{F_\bullet}(K_F).
$$
This shows that the expected inclusion in the question of the introduction does not hold in general.
\end{example}

In the above example, we observe that
$$
\deg g|_F\cdot \okval_{S_\bullet}(K_S) \supseteq  \okval_{C_\bullet}(K_C) \times \okval_{F_\bullet}(K_F).
$$
Although our expectation in the introduction may fail, we wonder whether the following question would have an affirmative answer, which would imply the Iitaka conjecture.

\begin{question}\label{question}
Let $f\colon X\to Y$ be an algebraic fiber space with general fiber $F$ over a point $\eta\in Y$. Assume that $K_Y$ and $K_F$ are effective. Let $X_\bullet$ be a fiber-type admissible flag on $X$ associated to $Y_\bullet, F_\bullet$, where $Y_\bullet$ is an admissible flag on $Y$ centered at $\eta\in Y$ containing a Nakayama subvariety $Y'$ of $K_Y$ and $F_\bullet$ is an admissible flag on $F$ centered at a general point of $F$ containing a Nakayama subvariety $F'$ of $K_F$.
Then do there exist some positive numbers $\alpha,\beta,\gamma>0$ such that
$$
\alpha\okval_{X_\bullet}(K_X) \supseteq \beta\okval_{Y_\bullet}(K_Y) + \gamma\okval_{F_\bullet}(K_F)
$$
hold?
\end{question}

In Example \ref{ex:inverseinclusion}, we can take $\alpha=\deg g|_F$, $\beta=\gamma=1$. By \cite[Theorem C]{CPW-okbdab}, using the abundance of the canonical divisors, we can find a positive integer $\alpha>0$ such that  $\Delta_{X_\bullet}^{\lim}(K_X)=\alpha\cdot\okval_{X_\bullet}(K_X)$. Such an integer $\alpha$ is given by the degree of the canonical contraction on a minimal model restricted on the image of the Nakayama subvariety. See \cite{CPW-okbdab} for more details. Assuming the abundance conjecture \cite[Conjecture 3.8]{BDPP}, we have an affirmative answer to Question \ref{question} by Theorem \ref{thm:oklimK}.

%%%%%%%%%%%%%%%%%%%%%%%%%%%%%%%%%%%%%%%%%%%%%%%%%%%%%

\end{document}